\documentclass[11pt,reqno]{amsart}
\usepackage{hyperref}
\usepackage{xcolor}
\hypersetup{
    colorlinks,
    linkcolor={red!50!black},
    citecolor={blue!50!black},
    urlcolor={blue!80!black}
}
\usepackage[nameinlink]{cleveref}
\usepackage{anysize} \marginsize{1.3in}{1.3in}{1in}{1in}
\usepackage{amsfonts}
\let\oldtocsection=\tocsection

\let\oldtocsubsection=\tocsubsection

\let\oldtocsubsubsection=\tocsubsubsection

\renewcommand{\tocsection}[2]{\bfseries\hspace{0em}\oldtocsection{#1}{#2}}
\renewcommand{\tocsubsection}[2]{\hspace{1em}\oldtocsubsection{#1}{#2}}
\renewcommand{\tocsubsubsection}[2]{\itshape\hspace{2em}\oldtocsubsubsection{#1}{#2}}

\usepackage{comment}
\usepackage[all,cmtip]{xy}
\usepackage{amsmath}
\usepackage{enumitem}
\usepackage{amsthm}
\usepackage{amssymb}
\usepackage{mathrsfs}
\usepackage{enumitem}
\usepackage{tikz-cd}
\usepackage{tikz}
\tikzset{
    labl/.style={anchor=south, rotate=90, inner sep=.5mm}
}

\usepackage[normalem]{ulem}

\newtheorem{thm}{Theorem}[section]

\newtheorem{prop}[thm]{Proposition}
\newtheorem{lemma}[thm]{Lemma}

\newtheorem{lem}[thm]{Lemma}

\theoremstyle{definition}

\newtheorem*{definition*}         {Definition}

\theoremstyle{remark}

\newtheorem{remark}[thm]{Remark}

\newcommand{\Hom}{\mathrm{Hom}}

\newcommand{\Z}{\mathbb{Z}}

\newcommand{\C}{\mathbb{C}}

\newcommand{\arrow}{\rightarrow}

\def\cC{\mathcal{C}}

\def\id{\operatorname{id}}

\def\Hom{\operatorname{Hom}}

\usepackage{amscd,amssymb,amsmath}
\usepackage{color}

\usepackage{mathrsfs}
\newcommand{\Dmod}{\mathscr{D}}
\newcommand{\Mmod}{\mathcal{M}}
\newcommand{\shO}{\mathscr{O}}
\newcommand{\shV}{\mathcal{V}}
\newcommand{\shT}{\mathscr{T}}
\newcommand{\Bt}{\tilde{B}}

\newcommand{\gt}{\tilde{g}}

\newcommand{\shVt}{\tilde{\shV}}

\newcommand{\derR}{\mathbf{R}}
\newcommand{\derL}{\mathbf{L}}
\newcommand{\shQ}{\mathscr{Q}}
\newcommand{\Ytl}{\tilde{Y}}
\newcommand{\htl}{\tilde{h}}
\newcommand{\xit}{\tilde{\xi}}
\newcommand{\etat}{\tilde{\eta}}
\DeclareMathOperator{\cont}{int}
\newcommand{\Vfix}{V^{\mathit{fix}}}
\newcommand{\shVfix}{\shV^{\mathit{fix}}}
\newcommand{\Pt}{P}

\DeclareMathOperator{\RatCurves}{RatCurves}
\newcommand{\shTfix}{\shT_{s(\cC)}^{\mathit{fix}}}
\newcommand{\shTpos}{\shT_{s(\cC)}^{\mathit{pos}}}
\newcommand{\shCfix}{\shT_{\cC}^{\mathit{fix}}}
\newcommand{\shCpos}{\shT_{\cC}^{\mathit{pos}}}
\newcommand{\shUPfix}{\shT_{M_0 \times \P^1}^{\mathit{fix}}}
\newcommand{\shUPpos}{\shT_{M_0 \times \P^1}^{\mathit{pos}}}
\def\cR{\mathcal{R}}
\newcommand{\Gt}{\tilde{G}}

\title[A Hodge-theoretic proof of Hwang's theorem]%
{A Hodge-theoretic proof of Hwang's theorem on base manifolds of Lagrangian fibrations}

 \author[B. Bakker]{Benjamin Bakker}
 \address{\noindent B. Bakker:  Dept. of Mathematics, Statistics, and Computer Science, University of Illinois at Chicago, Chicago, USA.}
\email{bakker.uic@gmail.com}
  \author[Ch. Schnell]{Christian Schnell}
\address{\noindent Ch. Schnell:  Department of Mathematics, Stony Brook University, Stony Brook, NY, USA.}
\email{christian.schnell@stonybrook.edu}

\def\P{\mathbb{P}}

\def\R{\mathbb{R}}
\def\Q{\mathbb{Q}}

\def\shW{\mathcal{W}}

\begin{document}
\maketitle
\begin{abstract}
	 We give a Hodge-theoretic proof of Hwang's theorem, which says that if
	 the base of a Lagrangian fibration of an irreducible holomorphic symplectic
	 manifold is smooth, it must be projective space.
\end{abstract}

\section{Introduction}
Let $X$ be an irreducible holomorphic symplectic manifold of dimension $2n$, and let
$f:X\to B$ be a Lagrangian fibration over a base of dimension $n$. In this
note, we reprove the following well-known result of Hwang \cite{Hwang:base}.

\begin{thm}[Hwang] \label{main}
	Let $f:X\to B$ be a Lagrangian fibration of an irreducible holomorphic symplectic
	manifold of dimension $2n$.  If $B$ is smooth, then $B\cong \P^n$.
\end{thm}

The proof that we present below is inspired by Hwang's original argument, but perhaps
clarified by the consistent use of Hodge theory, especially the theory of
variations of Hodge structure. There are two other important inputs. The first is the
following enhancement of Mori's characterization of projective space, due to Cho,
Miyaoka, and Shepherd-Barron \cite{cmsb} and Kebekus \cite{kebekus}.

\begin{thm}[Cho--Miyaoka--Shepherd-Barron, Kebekus] \label{thm:Mori} 
	Let $Y$ be a uniruled smooth projective variety and suppose there is a dense
	Zariski open subset $Y^{\circ} \subseteq Y$ such that every nonconstant rational
	curve $g \colon C\to Y$ meeting $Y^{\circ}$ satisfies $-\deg g^*\omega_Y\geq \dim
	Y+1$.  Then $Y$ is isomorphic to projective space.
\end{thm}

The second input is a result of Voisin \cite{voisintorsion} about the
variation of Hodge structure on the first cohomology of the fibers of a Lagrangian fibration.
Let $B^\circ\subseteq B$ be the open set over which the Lagrangian fibration
$f$ is smooth, and $f^\circ:X^\circ\to B^\circ$ the base change of $f$ to $B^\circ$.
It is known that the discriminant locus $D = B \setminus B^{\circ}$ is a nonempty
divisor. We denote by $V=R^1 f^\circ_*\Q_{X^\circ}$ the polarized variation of Hodge
structure of weight $1$ on $B^{\circ}$; by $\shV$ the underlying flat vector
bundle; by $\nabla$ the Gauss-Manin connection on $\shV$; and by $F^1 \shV =
f^{\circ}_* \Omega_{X^{\circ}/B^{\circ}}^1$ the Hodge bundle of type $(1,0)$.

\begin{thm}[Voisin] \label{thm:voisin} 
	$V \otimes \R$ is irreducible as a real variation of Hodge structure.
\end{thm}
The anticanonical degree of a curve (as well as its deformation theory) is linked to
the variation $V$ since contraction with the symplectic form yields a natural
isomorphism $\shT_{B^\circ}\cong F^1\shV$.  Work of Matsushita
\cite{Matsushita:higher} extends this to an isomorphism between $\shT_B$ and the
``canonical extension'' of $F^1\shV$ in codimension
one, and this will be an important ingredient in our approach.  

It is known by \cite{matsushitafano} that $B$ is a Fano manifold (and therefore
uniruled).  The strategy is to
assume the condition in \Cref{thm:Mori} is not satisfied, so that there is a low
degree rational curve meeting $B^\circ$, and conclude from this that:  (1) the curve
deforms to cover
$B$; and (2) the variation $V$ must have a nontrivial fixed part when restricted to
$C$.  We then obtain a nontrivial splitting of $V$ on the universal family of curves
and thus on a finite cover of $B$. To finish, we show that this splitting descends to $B^\circ$,
where it of course contradicts \Cref{thm:voisin}.  In our approach, the first step is proved using a Hodge module version of
Matsushita's result combined with a straightforward functoriality property of Hodge
modules to show that $g^*\shT_B$ is semi-positive (see \Cref{prop:semipositive}) and
that its maximal trivial quotient comes from the fixed part of the variation (see
\Cref{prop:fixed-part}).  As the low degree curve which is assumed to exist may 
a priori not deform, it can intersect the discriminant divisor in its singular
locus, and the Hodge module perspective provides a version of Matsushita's theorem
that works in higher codimension.  The descent step is straightforward in the case
that $f$ is non-isotrivial (using the result of \cite{Bakker:Matsushita}, see \Cref{rmk:nonisotrivial is done}), but we do not actually use this.  Instead, we prove the descent in general using a rather tricky argument which is a reinterpretation of an argument from Hwang's paper using Hodge theory to simplify the steps.  Note that the descent step must in some sense be subtle---in fact, it is the most difficult part of our proof---as it is
easy to construct examples where $V$ splits off a factor on a finite cover that is
constant on a covering family of rational curves (see \Cref{example}).  

\begin{remark}
The recent preprint by Li and Tosatti \cite{Li-Tosatti} contains another proof of
Hwang's theorem that follows the same overall strategy, but has a more analytic
flavor. Instead of Hodge theory, Li and Tosatti rely on the theory of special K\"ahler
manifolds \cite{Freed}. They use the curvature of the special K\"ahler metric on
$B^{\circ}$ to show that the restriction of $\shT_B$ to a rational curve of minimal
degree is semipositive, and that any trivial quotient of $g^* \shT_B$ actually
comes from a subbundle that is flat with respect to the Chern connection of the
metric. For the descent step at the end of the proof, however, they depend on one of
the results from our paper.
\end{remark}

\subsection*{Outline}In \Cref{sect tangent bundle} we prove the results on the semipostivity of $g^*\shT_B$ and the relation between the maximal trivial quotient and the fixed part of the variation.  We give two proofs:  the first in the projective case is more geometric, using results of Koll\'ar; the second as a consequence of the functoriality of Hodge module extensions.  In  \Cref{sect foliations} we make some general remarks relating subvariations to foliations on $B$.  In \Cref{sect deformation} we use the existence of a low degree curve to construct a nontrivial splitting of the variation on a finite cover.  In \Cref{sect connected comp}, we study the geometry of the subvarieties swept out by connected components of the family of rational curves through a given point. (For the sake of exposition, we reprove all the necessary results, which are due to Araujo \cite{Araujo} and Hwang \cite{Hwang:pre-base}, even though this makes the paper longer.)   In \Cref{sect cover}, we describe the main ingredient for descending the splitting to $B$, using ideas from \cite{Hwang:base}. Finally, we prove \Cref{main} in \Cref{sect proof}.

\subsection*{Notation and terminology}

By a \emph{rational curve} $g:C\to B$ we always mean a nonconstant morphism from $C\cong
\P^1$.  If the image meets the open subset $B^\circ$, we denote by
$g^\circ:C^\circ:=g^{-1}(B^\circ)\to B^\circ$ the restriction of $g$, and in general
the superscript ``$\circ$" always stands for objects that are related to $B^\circ$.  
We denote by $V_C:=(g^\circ)^*V$ the pullback of
the variation, and likewise for any pullback.  We typically use Roman letters $V,W$
to denote local systems and script letters $\shV,\shW$ to denote the
associated flat vector bundles.

\section{Restriction of the tangent bundle to rational curves}\label{sect tangent bundle}

In this section, we use results from Hodge theory to analyze the splitting behavior
of the tangent bundle $\shT_B$ on rational curves in $B$ that meet the smooth locus
$B^{\circ}$ of the Lagrangian fibration $f \colon X \to B$. Here it will be more
convenient to work with the variation of Hodge structure on $R^{2n-1} f_*^{\circ}
\Q_{X^{\circ}}$ instead of the one on $R^1 f_*^{\circ} \Q_{X^{\circ}}$. This
makes no essential difference, because the Hard Lefschetz theorem says that, after
tensoring by $\R$, the two variations are isomorphic up to a Tate twist by
$\R(-n+1)$; the isomorphism is induced by cup product with the $(n-1)$-th power of a
K\"ahler form on $X$. We will therefore change the notation slightly, and---in this
section only---use the symbol $V$ for the variation of Hodge structure on $R^{2n-1}
f_*^{\circ} \Q_{X^{\circ}}$. The relevant Hodge bundle is then $F^n \shV = R^{n-1}
f_*^{\circ} \omega_{X^{\circ}/B^{\circ}}$; it is of course isomorphic to $f_*^{\circ}
\Omega_{X^{\circ}/B^{\circ}}^1$.

Our starting point is the following result by Matsushita \cite{Matsushita:higher},
which describes the tangent bundle of $B$ in terms of Hodge theory.

\begin{thm} \label{thm:Matsushita}
	There is an isomorphism $R^{n-1} f_* \omega_{X/B} \cong \shT_B$, whose restriction
	to $B^{\circ}$ is the natural isomorphism induced by the holomorphic symplectic
	form.
\end{thm}

\begin{proof}
	Matsushita proves that the natural isomorphism on $B^\circ$ extends to $R^{n-1}
	f_* \omega_X \cong \Omega_B^{n-1} \cong \omega_B \otimes \shT_B$. Tensoring both
	sides by $\omega_B^{-1}$ then gives the desired result. Matsushita only states the
	theorem for projective $X$, but his proof carries over to the case
	where $X$ is a compact K\"ahler manifold. For an alternative proof, see
	\cite[\S4]{lagrangian}.
\end{proof}

Let $g \colon C \to B$ be a nonconstant morphism from $C \cong \P^1$ whose image
intersects $B^{\circ}$. Like any vector
bundle on $\P^1$, the pullback $g^* \shT_B$ decomposes as
\begin{equation} \label{eq:ai}
	g^* \shT_B \cong \shO(a_1) \oplus \dotsb \oplus \shO(a_n)
\end{equation}
for certain integers $a_1, \dotsc, a_n$. Our first result is that $g^* \shT_B$ is
semi-positive. 

\begin{prop}\label{prop:semipositive}
The pullback $g^* \shT_B$ is semi-positive, in the sense that $a_1, \dotsc, a_n \geq
0$.
\end{prop}

We will first present a geometric proof in the case when $X$ is projective;
afterwards, we will give a Hodge-theoretic proof that also works when $X$ is a
compact K\"ahler manifold.  Let us denote by $h \colon Y \to C$ the base change of
the Lagrangian fibration, and
let $r \colon \Ytl \to Y$ be a resolution of singularities of $Y$ that is an
isomorphism over the preimage of $C^{\circ} = g^{-1}(B^{\circ})$. For dimension
reasons, the composition $\htl = r \circ h \colon \Ytl \to C$ is still flat of
relative dimension $n$. We get a commutative diagram
\[
	\begin{tikzcd}
		\Ytl \rar{r} \drar[bend right=20]{\htl} & Y \dar{h} \rar & X \dar{f} \\
						& C \rar{g} & B
	\end{tikzcd}
\]
in which the square is Cartesian. The key step in the proof is the following lemma.

\begin{lem} \label{lem:morphism}
	There is a morphism of sheaves
	\[
		R^{n-1} \htl_* \omega_{\Ytl/C} \to g^* \shT_B
	\]
	that restricts to the natural isomorphism over the open subset $C^{\circ} =
	g^{-1}(B^{\circ})$.
\end{lem}

\begin{proof}
	Before giving the proof, we need to review a few facts about relative dualizing
	sheaves and base change. Here we probably have to assume that $X$ is projective
	(although it is known by \cite{Campana:LocalProjectivity} that Lagrangian
	fibrations are projective locally in the analytic topology on the base). In this algebraic setting, a good reference for duality theory is
	\cite[\href{https://stacks.math.columbia.edu/tag/0DWE}{Tag ODWE}]{stacks-project}. 
	To begin with, $B$ is smooth and $f$ is flat,
	and so all fibers of $f$ are local complete intersections in $X$ and therefore
	Gorenstein. Since $f$ is proper and flat, it follows that the relative dualizing
	sheaf $\omega_{X/B}$ is isomorphic to $\omega_X \otimes f^* \omega_B^{-1}$, and
	that the relative dualizing complex is $\omega_{X/B}[n]$. Moreover, in this
	situation, the relative dualizing sheaf commutes with arbitrary base change.

	Now we begin the proof. Flatness of $f$ and the base change property for the
	relative dualizing sheaf give us an isomorphism
	\[
		g^* \derR f_* \omega_{X/B} \cong \derR h_* \omega_{Y/C}.
	\]
	Since all the higher direct images $R^i f_* \omega_{X/B}$ are locally free (by
	Matsushita's theorem), we get in particular that
	\begin{equation} \label{eq:base-change-iso}
		g^* \shT_B \cong g^* R^{n-1} f_* \omega_{X/B} \cong R^{n-1} h_* \omega_{Y/C}.
	\end{equation}
	Consider now the resolution of singularities $r \colon \Ytl \to Y$. The trace map
	gives us a morphism $r_* \omega_{\Ytl} \to \omega_Y$ that is an isomorphism over
	the smooth locus of $Y$ and therefore injective. The induced morphism
	\begin{equation} \label{eq:trace}
		R^{n-1} \htl_* \omega_{\Ytl/C} \to R^{n-1} h_* \omega_{Y/C}
	\end{equation}
	is of course still an isomorphism over $C^{\circ}$. We then get the desired
	morphism by composing \eqref{eq:trace} and \eqref{eq:base-change-iso}; it is
	clearly the natural isomorphism over $C^{\circ}$.
\end{proof}

We can now prove that $g^* \shT_B$ is always semi-positive.

\begin{proof}[Proof of \Cref{prop:semipositive}]
	From \Cref{lem:morphism}, we get a short exact sequence
	\[
		0 \to R^{n-1} \htl_* \omega_{\Ytl/C} \to g^* \shT_B \to \shQ \to 0,
	\]
	in which $\shQ$ is supported on the finite set $C \setminus C^{\circ}$ and therefore
	a torsion sheaf. This reduces the problem to the semi-positivity of $R^{n-1} \htl_*
	\omega_{\Ytl/C}$. Since $\htl \colon \Ytl \to C$ is a K\"ahler fiber space over a
	curve, this is the content of a classical result by Fujita \cite{Fujita}. Alternatively, one can argue using Koll\'ar's results \cite{Kollar:DualizingI} about higher direct images of dualizing sheaves. Since $C \cong \P^1$, we have 
    \[
        R^i h_* \omega_{\Ytl/C} \otimes \shO(-1)
        \cong R^i h_* \omega_{\Ytl} \otimes \shO(1),
    \]
    and this bundle has no higher cohomology by Koll\'ar's vanishing theorem (which
	 applies to any morphism from a compact K\"ahler manifold to a projective
	 variety). But on $\P^1$, this is only possible if $R^i h_* \omega_{\Ytl/C}$ is
	 semi-positive.     
\end{proof}

We also need to understand the trivial part in the decomposition \eqref{eq:ai}, at
least in the case of generic rational curves. Let us therefore assume from now on
that $g \colon C \to B$ is an immersed rational curve that intersects the
discriminant divisor $D = B \setminus B^{\circ}$ transversely. The 
preimage $Y = X \times_C B$ is then already smooth, and \eqref{eq:base-change-iso} gives
us an isomorphism
\[
	g^{\ast} \shT_B \cong R^{n-1} h_* \omega_{Y/C}.
\]
Suppose that in the decomposition \eqref{eq:ai}, we have $a_1, \dotsc, a_r = 0$ and
$a_{r+1}, \dotsc, a_n > 0$. In that case, the bundle $g^{\ast} \shT_B$ has a
canonical trivial quotient
\[
	g^{\ast} \shT_B \to \shO_C^{\oplus r}.
\]
On the open subset $C^{\circ} = g^{-1}(B^{\circ})$, we have a polarized variation of
Hodge structure of weight $2n-1$, whose underlying local system is $R^{2n-1} h^{\circ}_*
\Q_{Y^{\circ}}$; here $h^{\circ} \colon Y^{\circ} \to C^{\circ}$ is the restriction of $h$. Our next result identifies the maximal trivial quotient of $g^{\ast} \shT_B$
as coming from the fixed part of the variation of Hodge structure on $R^{2n-1} h^{\circ}_*
\Q_{Y^{\circ}}$.

\begin{prop} \label{prop:fixed-part}
	The fixed part of the variation of Hodge structure on $R^{2n-1} h^{\circ}_* \Q_{Y^{\circ}}$ has rank
	exactly $2r$, and the Hodge bundle of type $(n,n-1)$ of the fixed part projects
	isomorphically to the maximal trivial quotient of $g^{\ast} \shT_B$.
\end{prop}

\begin{proof}
	To simplify the notation, let us set $V_C = R^{2n-1} h^{\circ}_* \Q_{Y^{\circ}}$,
	and let us denote by $\shV_C$ the underlying flat vector bundle, and by $F^n
	\shV_C$ the Hodge bundle of
	type $(n,n-1)$. Since $h$ has relative dimension $n$, the restriction of $R^{n-1}
	h_* \omega_{Y/C}$ to the open subset $C^{\circ}$ is canonically isomorphic to $F^n
	\shV_C$. In fact, more is true. Let $\shVt_C^{>-1}$ denote Deligne's canonical
	extension of the flat vector bundle $\shV_C$; it is uniquely characterized by the
	condition that the connection on $\shV_C$ has logarithmic poles on $\shVt_C^{>-1}$ and
	that the residues of the connection at each point of $C \setminus C^{\circ}$ have
	eigenvalues contained in the interval $(-1,0]$. With this notation, Koll\'ar
	\cite{Kollar:DualizingI} proves that 
	\begin{equation} \label{eq:extension}
		R^{n-1} h_* \omega_{Y/C} \cong \shVt_C^{>-1} \cap j_*(F^n \shV_C), 
	\end{equation}
	where $j \colon C^{\circ} \to C$ is the inclusion. 

	The fixed part of $V_C$ is a $\Q$-Hodge structure of type $(n,n-1) +
	(n-1,n)$, and so its dimension is an even number, say $2s$. The fixed part then
	contributes a trivial summand of rank $2s$ to the canonical extension
	$\shVt_C^{>-1}$, and, due to \eqref{eq:extension}, a trivial summand of rank $s$ to
	the bundle $R^{n-1} h_* \omega_{Y/C}$. This gives $r \geq s$. To prove that $r
	= s$, we need to argue that any nontrivial morphism of the form
	\begin{equation} \label{eq:morphism-to-O}
		R^{n-1} h_* \omega_{Y/C} \to \shO_C
	\end{equation}
	must come from the fixed part of $V_C$. This is a straightforward computation with
	duality. First, we observe that since $C$ is a curve, the derived pushforward of
	the line bundle $\omega_{Y/C}$ decomposes as
	\[
		\derR h_* \omega_{Y/C} \cong \bigoplus_{i=0}^n R^i h_* \omega_{Y/C}[-i]
	\]
	in the derived category of coherent sheaves on $C$. (This kind of result is due to
	Koll\'ar \cite{Kollar:DualizingII} in general, but only needs elementary
	homological algebra in the case of curves.) Consequently, $R^{n-1} h_*
	\omega_{Y/C}$ is a direct summand of $\derR h_* \omega_{Y/C}[n-1]$. 
		
	Grothendieck duality for the morphism $h \colon Y \to C$ yields
	\[
		\Hom_C \Bigl( \derR h_* \omega_{Y/C}[n-1], \shO_C \Bigr)
		\cong \Hom_Y \Bigl( \omega_{Y/C}[n-1], \omega_{Y/C}[n] \Bigr)
		\cong H^1(Y, \shO_Y).
	\]
	After substituting in the above formula for the direct image and remembering that
	$H^1(C, \shO_C) = 0$ (because $C \cong \P^1$), we get an isomorphism
	\[
		\Hom_C \Bigl( R^{n-1} h_* \omega_{Y/C}, \shO_C \Bigr) \cong H^1(Y, \shO_Y).
	\]
	Now $Y$ is a compact K\"ahler manifold, and so by classical Hodge theory, the
	morphism $H^1(Y, \C) \to H^1(Y, \shO_Y)$ is surjective. After some
	compatibility checking, it follows that any morphism as in \eqref{eq:morphism-to-O} comes
	about in the following way: at any point $c \in C^{\circ}$, the restriction of
	\eqref{eq:morphism-to-O} is the linear functional $H^{n,n-1}(Y_c) \to \C$
	obtained by integration against the restriction of a fixed class in $H^1(Y, \C)$
	to the fiber $Y_c = h^{-1}(c)$. Since restriction to the smooth fibers of $h
	\colon Y \to C$ determines a morphism of variations of Hodge structure
	\[
		H^1(Y, \Q) \otimes R^{2n-1} h^{\circ}_* \Q_{Y^{\circ}} 
		\to R^{2n} h^{\circ}_* \Q_{Y^{\circ}} \cong \Q_C(-n),
	\]
	we conclude (by the semi-simplicity of polarized variations of Hodge structure)
	that the fixed part of $R^{2n-1} h^{\circ}_* \Q_{Y^{\circ}}$ must have rank at least $2r$.
\end{proof}

\begin{remark}
	It is possible that the construction with relative dualizing sheaves and base
	change still works in the analytic setting, but we have not been able to
	find a reference for the necessary technical results.
\end{remark}

In the remainder of this chapter, we explain how one can modify the argument above to
make it work when $X$ is just a compact K\"ahler manifold.  The proof relies
on a few basic properties of Hodge modules \cite{Saito:HM,Saito:MHM}. Let $M$ denote
the polarized Hodge module of weight $\dim B + (2n-1)$ with strict support $B$,
associated to the variation of Hodge structure $V = R^{2n-1} f_*^{\circ} \Q_{X^{\circ}}$
on $B^{\circ}$ by \cite[Thm.~3.21]{Saito:MHM}.  Let $\Mmod$ denote the (regular
holonomic) left $\Dmod_B$-module underlying $M$, and let $F_{\bullet} \Mmod$ be its
Hodge filtration. By Saito's direct image theorem \cite[Thm.~5.3.1]{Saito:HM} (and
\cite{Saito:Kollar} for the constant Hodge module on a compact K\"ahler manifold), we
have $F_{-n-1} \Mmod = 0$, and
\[
	\omega_B \otimes F_{-n} \Mmod \cong R^{n-1} f_* \omega_X.
\]
Together with \Cref{thm:Matsushita}, this implies that $F_{-n} \Mmod \cong \shT_B$,
and so we can also think of the tangent sheaf as being part of a Hodge module on $B$.
We may sometimes refer to the coherent $\shO_B$-module $F_{-n} \Mmod$ as the ``Hodge
module extension'' of $F^n \shV$.

There is an alternative description of the coherent sheaf $F_{-n} \Mmod$ using
resolution of singularities. Let $\mu \colon \Bt \to B$ be a log resolution of the
discriminant divisor $D = B \setminus B^{\circ}$ that is an isomorphism over the open
subset over which $D$ is a divisor with normal crossings, hence in particular
over $B^{\circ}$. Let $\shV$ denote the flat vector bundle underlying the variation
of Hodge structure $V = R^{2n-1} f_*^{\circ} \Q_{X^{\circ}}$, and let $F^n \shV =
R^n f_*^{\circ} \omega_{X^{\circ}/B^{\circ}}$ be the Hodge bundle of type
$(n,n-1)$. As before, let $\shVt^{>-1}$ denote Deligne's canonical extension of the
flat vector bundle $\shV$; this is the unique extension of $\shV$ to a vector bundle
on $\Bt$ on which the connection has logarithmic poles with residues having
eigenvalues in the interval $(-1,0]$. By Schmid's nilpotent orbit theorem, the Hodge
bundle $F^n \shV$ extends to a subbundle $F^n \shVt^{>-1} = \shVt^{>-1} \cap j_*(F^n
\shV)$, where $j \colon B^{\circ} \to \Bt$ is the inclusion. According to
\cite[Thm.~2.4 and (2.2.2)]{Saito:Kollar}, we then have
\begin{equation} \label{eq:resolution}
	\omega_B \otimes F_{-n} \Mmod 
	\cong \derR \mu_* \Bigl( \omega_{\Bt} \otimes F^n \shVt^{>-1} \Bigr).
\end{equation}
This morphism is an isomorphism over the open subset where $D = B \setminus
B^{\circ}$ is a divisor with normal crossings.

We can now present the alternative proof of \Cref{prop:semipositive}. Let $g \colon C
\to B$ be a rational curve meeting the smooth locus $B^{\circ}$ of the Lagrangian
fibration. Let $V_C$ be the variation of Hodge structure obtained by pulling back
$V = R^{2n-1} f^{\circ}_{\ast} \Q_{X^{\circ}}$ to the open subset $C^{\circ} =
g^{-1}(B^{\circ})$, and let $M_C$ be its extension to a polarized Hodge module of
weight $\dim C + (2n-1)$ with strict support $C$.  Let $(\Mmod_C, F_{\bullet}
\Mmod_C)$ be the underlying filtered left $\Dmod_C$-module. If we again write
$\shVt_C^{>-1}$ for Deligne's canonical extension, then for the same reason as above,
we have
\begin{equation} \label{eq:canext}
	F_{-n} \Mmod_C \cong F^n \shVt_C^{>-1}, 
\end{equation}
Since $V_C$ is a polarized variation of Hodge structure with quasi-unipotent local
monodromy on a curve, this vector bundle is known to be semi-positive
\cite{Peters}.\footnote{Peters claims to be working with the canonical extension for
the interval $[0,1)$, but his signs in the formula for the Hodge norm are wrong, and
so he should really be using the interval $(-1,0]$.} Alternatively, using the
notation from the proof of \Cref{prop:semipositive}, we can observe that $F^n
\shVt_C^{>-1} \cong R^{n-1} h_* \omega_{\Ytl/C}$, which is semi-positive for the
reason explained above. Either way, to relate this semi-positivity to that of $g^*
\shT_B$, we need the following variant of our earlier \Cref{lem:morphism}.

\begin{lem} \label{lem:morphism-MHM}
	There is a morphism of sheaves $F_{-n} \Mmod_C \to g^* \bigl( F_{-n}
	\Mmod \bigr)$ that restricts to the natural isomorphism over $C^{\circ} =
	g^{-1}(B^{\circ})$.
\end{lem}

\begin{proof}
	This kind of morphism is constructed in \cite[\S13-14]{Schnell:Neron} in much
	greater generality. We present here a more elementary argument
	that reduces the input from the theory of Hodge modules to a minimum. Recall
	from \eqref{eq:resolution} that we have an isomorphism
	\[
		\derR \mu_* \Bigl( \omega_{\Bt} \otimes F^n \shVt^{>-1} \Bigr)
		\cong \omega_B \otimes F_{-n} \Mmod.
	\]
	The right-adjoint of the functor $\derR \mu_*$ is the exceptional inverse image
	functor $\mu^!  = \omega_{\Bt} \otimes \mu^* \omega_B^{-1} \otimes \derL \mu^*$. After
	applying this to the above isomorphism, we get a morphism
	\[
		F^n \shVt^{>-1} \to \derL \mu^* \bigl( F_{-n} \Mmod \bigr)
	\]
	in the derived category; since $F_{-n} \Mmod \cong \shT_B$ is locally free,
	this is a morphism of sheaves
	\begin{equation} \label{eq:morphism1}
		F^n \shVt^{>-1} \to \mu^* \bigl( F_{-n} \Mmod \bigr).
	\end{equation}
	By construction, this morphism is an isomorphism outside the exceptional divisor
	of $\mu$, and therefore injective.

	The rational curve $g \colon C \to B$ intersects the smooth locus $B^{\circ}$, and
	therefore lifts uniquely to a rational curve $\gt \colon C \to \Bt$, making the
	following diagram commute:
	\[
		\begin{tikzcd}
			C \rar{\gt} \drar[bend right=20]{g} & \Bt \dar{\mu} \\
							& B
		\end{tikzcd}
	\]
	After pulling back \eqref{eq:morphism1} along $\gt$, we obtain a morphism
	\[
		\gt^* \bigl( F^n \shVt^{>-1} \bigr) \to g^* \bigl( F_{-n} \Mmod \bigr).
	\]
	The properties of the canonical extension give us a natural morphism
	$\shVt_C^{>-1} \to \gt^* \bigl( \shVt^{>-1} \bigr)$ that is an isomorphism over
	$C^{\circ}$. Consequently, we have a morphism
	\[
		F^n \shVt_C^{>-1} \to \gt^* \bigl( F^n \shVt^{>-1} \bigr),
	\]
	and after composing the two morphisms and remembering the identity in
	\eqref{eq:canext}, we get the desired morphism
	\[
		F_{-n} \Mmod_C \to g^*(F_{-n} \Mmod).
	\]
	By construction, it is an isomorphism over the open subset $C^{\circ} =
	g^{-1}(B^{\circ})$.
\end{proof}

This gives us an injective morphism from the semi-positive vector bundle 
$F_{-n} \Mmod_C \cong F^n \shVt_C^{>-1}$ into $g^*(F_{-n} \Mmod) \cong g^* \shT_B$,
and since this morphism is an isomorphism over $C^{\circ}$, we can argue as before to
finish the proof of \Cref{prop:semipositive}. 

Now suppose that the rational curve $g \colon C \to B$ intersects the discriminant
divisor $D$ transversely. This means that it stays inside the open subset where $F_{-n}
\Mmod$ is described by the canonical extension, and because the intersection is
transverse, the morphism
\[
	F^n \shVt_C^{>-1} \to g^* \shT_B
\]
is an isomorphism. Moreover, the fiber product $Y = C \times_B X$ is again smooth,
and so by \eqref{eq:extension}, we have $F^n \shVt_C^{>-1} \cong R^{n-1} h_* \omega_{Y/C}$,
where $h \colon Y \to C$ is the base change of the Lagrangian fibration.
\Cref{prop:fixed-part} therefore still applies in this setting.

\section{Foliations and variations of Hodge structure}\label{sect foliations}

The purpose of this section is to prove a couple of small technical results that relate the
variation of Hodge structure on $V = R^1 f_*^{\circ} \Q_{X^{\circ}}$ to the theory of
foliations. Here $f \colon X \to B$ is again a Lagrangian fibration from a compact
K\"ahler manifold $X$. 

Recall that $\shV$ denotes the underlying flat vector bundle, $\nabla \colon \shV \to
\Omega_{B^{\circ}}^1 \otimes \shV$ the Gauss-Manin connection, and $F^1 \shV =
f_*^{\circ} \Omega_{X^{\circ}/B^{\circ}}^1$ the Hodge bundle. We have $F^1 \shV \cong
\shT_{B^{\circ}}$, with the isomorphism induced by contraction with the holomorphic
symplectic form $\sigma \in H^0(X, \Omega_X^2)$. We are going to need the following
result by Freed \cite{Freed} that describes the Lie bracket on the
tangent bundle in terms of the connection.

\begin{lemma} \label{lem:bracket}
	For any two holomorphic vector fields $\xi, \eta \in \shT_{B^{\circ}}$, one has
	\[
		\nabla_{\xi}(\eta) - \nabla_{\eta}(\xi) = [\xi, \eta].
	\]
\end{lemma}

\begin{proof}
	Let $U \subseteq B^{\circ}$ be an open subset and $\xi, \eta \in H^0(U, \shT_B)$
	be two holomorphic vector fields. Since $f^{\circ} \colon X^{\circ} \to B^{\circ}$
	is a fiber bundle, we can lift them to smooth vector fields $\xit, \etat \in A^0
	\bigl( f^{-1}(U), \shT_X \bigr)$; then $[\xit, \etat]$ is a smooth lifting of the
	Lie bracket $[\xi, \eta]$. The isomorphism $\shT_{B^{\circ}} \cong F^1 \shV$
	takes the holomorphic vector field $\xi$ to the fiberwise cohomology class of the
	smooth $(1,0)$-form 
	\[
		\Omega_{\xit} = \cont(\xit)(\sigma) \in A^{1,0} \bigl( f^{-1}(U) \bigr),
	\]
	obtained by contracting the vector field $\xit$ against the holomorphic symplectic
	form. In these terms, the Gauss-Manin connection $\nabla_{\xi}(\eta) \in H^0(U,
	\shV)$ is given by taking the fiberwise cohomology class of the smooth $1$-form
	\[
		\cont(\xit)(d \Omega_{\etat}) \in A^1 \bigl( f^{-1}(U) \bigr).
	\]
	Now we compute that, for any smooth vector field $\lambda$ on $f^{-1}(U)$, one has
	\begin{align*}
		\Bigl( &\cont(\xit)(d \Omega_{\etat}) - \cont(\etat)(d \Omega_{\xit}) \Bigr)(\lambda)
		= d \Omega_{\etat}(\xit, \lambda) - d \Omega_{\xit}(\etat, \lambda) \\
		&= \Bigl( \xit \cdot \sigma(\etat, \lambda) - \lambda \cdot \sigma(\etat, \xit) - \sigma
		\bigl( \etat, [\xit, \lambda] \bigr) \Bigr)
		- \Bigl( \etat \cdot \sigma(\xit, \lambda) - \lambda \cdot \sigma(\xit, \etat)
		- \sigma \bigl( \xit, [\etat, \lambda] \bigr) \Bigr) \\
		&= \sigma \bigl( [\xit, \etat], \lambda \bigr) + \lambda \cdot \sigma(\xit,
		\etat),
	\end{align*}
	using the coordinate-free formula for the exterior derivative, and the identity
	\begin{align*}
		0 = (d \sigma)(\lambda, \xit, \etat) &= \lambda \cdot \sigma(\xit, \etat) + \xit
		\cdot \sigma(\etat, \lambda) + \etat \cdot \sigma(\lambda, \xit) \\
		&+ \sigma \bigl( \lambda, [\xit, \etat] \bigr) + \sigma \bigl( \xit, [\etat,
		\lambda] \bigr) + \sigma \bigl( \etat, [\lambda, \xit] \bigr).
	\end{align*}
	This proves that $\cont(\xit)(d \Omega_{\etat}) - \cont(\etat)(d \Omega_{\xit}) =
	\Omega_{[\xit, \etat]} + d \bigl( \sigma(\xit, \etat) \bigr)$, which then gives the desired
	identity in the first cohomology of the fibers of the Lagrangian fibration.
\end{proof}

\begin{remark}
Since $\shV^{1,0} \cong \shT_{B^{\circ}}$ and $\shV^{0,1} \cong
\Omega_{B^{\circ}}^1$, this also implies that the Higgs field
\[
	\overline{\nabla} \colon \shT_{B^{\circ}} \otimes \shV^{1,0} \to \shV^{0,1}
\]
is symmetric in its two arguments. This fact is a special case of the ``cubic
condition'' of Donagi and Markman \cite{Donagi+Markman}. The identity in
\Cref{lem:bracket} is actually equivalent to the statement that the hermitian metric
on $B^{\circ}$ induced by the Hodge metric on the bundle $\shV^{1,0}$ is a K\"ahler
metric \cite{Freed}.
\end{remark}

The formula for the Lie bracket allows us to produce foliations from
subvariations of $V$, using the fact that $F^1 \shV \cong \shT_{B^{\circ}}$ is the
tangent bundle of $B^{\circ}$.

\begin{lemma} \label{lem:foliation}
	Let $p \colon U \to B^{\circ}$ be a holomorphic mapping that is locally
	biholomorphic, and let $W$ be a subvariation of Hodge structure of $V_U = p^* V$.
	Then the distribution given by the Hodge bundle $F^1 \shW \subseteq F^1 \shV_U$ is
	integrable, hence defines a foliation on $U$.
\end{lemma}

\begin{proof}
	Since $p$ is locally biholomorphic, we have $F^1 \shV_U = p^*(F^1 \shV) \cong p^*
	\shT_{B^{\circ}} \cong \shT_U$. Because $W$ is a subvariation of Hodge structure,
	the vector bundle $\shW$ is preserved by the Gauss-Manin connection on $\shV_U$.
	From $\shW \cap F^1 \shV_U = F^1 \shW$, we get
	\[
		[\xi, \eta] = \nabla_{\xi}(\eta) - \nabla_{\eta}(\xi) \in H^0(U, \shW) \cap
		H^0(U, F^1 \shV_U) = H^0(U, F^1 \shW)
	\]
	for any two holomorphic vector fields $\xi, \eta$ in the subspace $H^0(U, F^1
	\shW) \subseteq H^0(U, \shT_U)$. This proves that the distribution is integrable.
\end{proof}

We are also going to need the following simple fact about the integrability of the
sum of two integrable distributions.

\begin{lemma} \label{lem:integrable-sum}
	Let $F$ and $G$ be holomorphic foliations of rank $d$ on a complex manifold,
	whose sum $F + G$ is a foliation of rank $2d$. Take a leaf $L$ of $F$, and
	form the union of the leaves of $G$ through points of $L$. Then the
	resulting submanifold is a leaf of $F+G$.
\end{lemma}

\begin{proof}
	Fix a point on $L$, and let $S$ be the unique leaf of $F+G$ through that point;
	then $S$ is a (locally defined) complex submanifold of dimension $2d$, and since
	$F \subseteq F+G$, it contains $L$. The restriction $G \vert_S$ is a foliation on
	$S$, and so $S$ is locally the union of leaves of $G$. For dimension reasons,
	those leaves of $G \vert_S$ that pass through points of $L$ fill up all of $S$,
	and this gives the desired result.
\end{proof}

\begin{remark}
The proof of the lemma shows that the construction is actually symmetric in $F$ and
$G$: if we start from a leaf of $G$ that is contained in $S$, and form the union
of the leaves of $F$ through its points, we get the same submanifold $S$ as
before.
\end{remark}

Lastly, we need to analyze the behavior of the tangent bundle at smooth points of the
discriminant locus $D$. It is known that $D$ is always a divisor in $B$
\cite[Prop.~3.1]{Hwang+Oguiso}, and that the local monodromy at any smooth point of
$D$ is always nontrivial \cite[Prop.~3.2]{Hwang:base}. Our proof does not actually
rely on these two facts, but it is nevertheless useful to keep them in mind. The next
result says that if the local monodromy at a smooth point of $D$ has finite order,
then it looks like the local monodromy in an isotrivial elliptic fibration on a K3
surface.

\begin{lemma} \label{lem:finite-monodromy}
	Suppose that, in a neighborhood of a given point $b \in B$, the discriminant locus
	$D$ is smooth of dimension $n-1$, and that the local monodromy around $D$ has
	finite order. Then the local monodromy transformation has exactly two nontrivial
	eigenvalues, which are a $4$th or $6$th root of unity and its complex conjugate.
\end{lemma}

\begin{proof}
	Let $U$ be a neighborhood of the point in question, and suppose that $D \cap U$ is
	smooth and defined by a nontrivial holomorphic function $h \in \Gamma(U, \shO_B)$. Let $V_U$
	denote the restriction of the variation of Hodge structure $V = R^1 f_*^{\circ}
	\Q_{X^{\circ}}$ to this open subset. By Matsushita's theorem, the isomorphism
	$\shT_{U^{\circ}} \cong F^1 \shV$ extends to an isomorphism
	\[
		\shT_U \cong F^1 \shVt_U^{>-1}
	\]
	between the tangent bundle of $U$ and the subbundle $F^1 \shVt_U^{>-1}
	= \shVt_U^{>-1} \cap j_*(F^1 \shV_U)$ of the canonical extension of $(\shV_U,
	\nabla)$ along the open embedding $j \colon U^{\circ} \to U$. 

	Since the monodromy transformation $T$ has finite order, all of its eigenvalues
	are roots of unity. Moreover, the weight
	filtration defined by the logarithm of the unipotent part of $T$ is trivial, and
	so the limiting mixed Hodge structure at $b$ is a polarized Hodge structure of
	type $(1,0) + (0,1)$, and $T$ is an automorphism of this Hodge structure
	\cite[Thm.~6.16]{Schmid}. By choosing a basis that is adapted to the Hodge
	decomposition and the eigenspace decomposition of $T$, we obtain a local frame
	\[
		v_1, \dotsc, v_n, v_{n+1}, \dotsc, v_{2n} 
			\in \Gamma \bigl( U, \shVt_U^{>-1} \bigr)
	\]
	for the canonical extension, with the property that 
	\[
		\nabla(v_j) = \alpha_j \frac{dh}{h} \otimes v_j,
	\]
	where $\alpha_1, \dotsc, \alpha_n \in \Q \cap (-1,0]$. As explained in
	\cite[\S2]{Cattani+Kaplan} (and in \cite{VHS-disk} for the case of complex variations of
	Hodge structure), we can then find a collection of holomorphic functions
	$f_{j,k}$ near the point $b$, such that the $n$ sections
	\[
		s_j = v_j + \sum_{k=n+1}^{2n} f_{j,k} v_k, \quad j=1, \dotsc, n,
	\]
	define a local frame for the bundle $F^1 \shVt_U^{>-1}$, after shrinking the open
	set $U$ if necessary. Because of the isomorphism $\shT_U \cong F^1 \shVt_U^{>-1}$,
	we may also view $s_1, \dotsc, s_n$ as holomorphic vector fields on $U$ that
	generate the tangent bundle $\shT_U$.

	We now apply the formula for the Lie bracket in \Cref{lem:bracket}. It gives
	\[
		[s_i, s_j] = \nabla_{s_i}(s_j) - \nabla_{s_j}(s_i) 
		\equiv \alpha_j \frac{dh(s_i)}{h} \cdot s_j - \alpha_i \frac{dh(s_j)}{h} \cdot
		s_i \mod v_{n+1}, \dotsc, v_{2n},
	\]
	and because of the special shape of the sections $s_1, \dotsc, s_n$, the
	congruence is strong enough to imply that
	\begin{equation} \label{eq:bracket-canonical-extension}
		[s_i, s_j] = \alpha_j \frac{dh(s_i)}{h} \cdot s_j 
			- \alpha_i \frac{dh(s_j)}{h} \cdot s_i.
		\end{equation}
	Now the left-hand side is again a holomorphic vector field on $U$, and so we
	deduce that, for every pair of indices $1 \leq i < j \leq n$, either $\alpha_i = 0$, or
	$h$ divides $dh(s_j)$. This last condition means that the vector field $s_j$ is
	tangent to the hypersurface $D \cap U$. Since not all the vector fields can be
	tangent to $D$, we conclude that at most one of $\alpha_1, \dotsc, \alpha_n$ can
	be nonzero; in fact, exactly one of them has to be nonzero, because the monodromy
	transformation $T$ is nontrivial.

	To prove the remaining assertion, we observe that $T$ is defined over $\Z$ and
	that $\det T = 1$. Therefore $T$ has exactly two nontrivial eigenvalues, which are
	roots of unity and complex conjugates of each other. By checking the possible
	values of Euler's $\phi$-function, it is then easy to see that the two eigenvalues
	must be either $4$th or $6$th roots of unity.
\end{proof}

\begin{remark}
	Suppose that $\alpha_1 \neq 0$ and $\alpha_2 = \dotsb = \alpha_n = 0$.
	The identity in \eqref{eq:bracket-canonical-extension} shows that the $n-1$
	holomorphic vector fields $s_2, \dotsc, s_n$ are tangent to $D$ and 
	commute with each other. This observation will play an important role later on.
\end{remark}

\begin{remark}
	By the same method, one can show that if $T = T_s \cdot e^{2 \pi i N}$ is the
	Jordan decomposition of the monodromy transformation (with $N^2 = 0$), then
	$\operatorname{rk} N \leq 1$, and the limiting mixed Hodge structure
	contains a pure Hodge structure of weight $1$ and dimension $\geq 2n-2$. This is
	related to the classification of the general singular fiber of a Lagrangian
	fibration in \cite{Hwang+Oguiso}.
\end{remark}

\section{Deformation theory and the universal family of curves}\label{sect deformation}

We first explain why \Cref{prop:semipositive} and \Cref{prop:fixed-part}, together with
some basic deformation theory, imply that $V$ splits off a factor universally
computing the fixed part of the variation of Hodge structure on any rational curve of
minimal degree, at least on a certain finite cover.  We recommend
Koll\'ar's book \cite{kollarrational} for the necessary background on rational
curves. Throughout, $B$ is the smooth base of a Lagrangian fibration, $B^{\circ}$ is
the complement of the discriminant locus, and $n = \dim B$.

Recall that $\omega^\vee_B = \det \shT_B$ is ample \cite{matsushitafano}.  Suppose $g_0:C_0\to B$
is a rational curve meeting $B^\circ$ and of minimal anticanonical degree among such
curves. We know from \Cref{prop:semipositive} that $g_0^{\ast} \shT_B$ is
semi-positive; consequently, we must have 
\begin{equation} \label{eq:degree}
	\deg g_0^{\ast} \shT_B = d+1
\end{equation}
for some integer $d \in \{1, \dotsc, n\}$. (If the degree is $\geq n+2$, then Mori's
Bend and Break technique produces a rational curve of smaller degree through every
point of $g_0(C)$, contradicting minimality.) In view of \Cref{thm:Mori}, our goal
will be to show that $d=n$. As $B$ is Fano, it is known that a very
general rational curve of minimal anticanonical degree deforms to cover $B$ and moreover the pullback of the tangent
bundle decomposes as
\begin{equation} \label{eq:generic-decomposition}
	g_0^{\ast} \shT_B \cong 
		\shO(2) \oplus \shO(1)^{\oplus (d-1)} \oplus \shO^{\oplus (n-d)}.
\end{equation}
By \Cref{prop:fixed-part}, the restriction of the variation of Hodge structure $V$ to
such a curve therefore has a fixed part of rank exactly $2(n-d)$.

We work with the (normalized) moduli space of rational curves $\RatCurves^n(B)$,
following Koll\'ar's notation \cite[II.2.5]{kollarrational}. Let $M$ be the open
subset of $\RatCurves^n(B)$
consisting of all rational curves of anticanonical degree $d$ that intersect
$B^{\circ}$ and belong to the same connected component as $g_0$. The universal family
$\pi \colon \cC \to M$ is a $\P^1$-bundle over $M$, and we denote by $G \colon \cC
\to B$ the evaluation morphism. We use the notation $g \colon C \to B$ for the
individual curves in our family; here $C$ is the fiber of $\pi \colon \cC \to M$ over
the point in question, and $g = G \vert_C$.

For a point $g:C\to B$ in our moduli space $M$, we have an exact sequence
\[
	0\to \shT_C\xrightarrow{dg} g^*\shT_B\to N_g \to 0
\] 
and the obstructions for $M$ lie in $H^1(C, N_g)$.  This group vanishes by
\Cref{prop:semipositive}, and so $M$ is smooth at such a point, of dimension 
\[
	\dim H^0(C, N_g) = \chi(C, N_g) = n+d-2.
\]
Moreover, as $g^*\shT_B$ is globally generated and by the
minimal degree condition, a generic deformation of our initial curve $g_0$ intersects
the discriminant locus $D = B\setminus B^\circ$ transversely
\cite[II.3.7]{kollarrational}.  Finally, for any $g \colon C \to B$ in the space $M$,
and for any point $c \in C$, we have a short exact sequence
\[
	\begin{tikzcd}
		0 \rar&  H^0(C, \shT_C) \rar{r_c-dg} & T_cC\oplus H^0(C, g^*\shT_B)
		\rar{\tau} & T_c \, \cC \rar & 0
	\end{tikzcd}
\]
where $r_c$ is restriction to $c$.  The map $\tau$ is naturally interpreted as the
derivative of the map $\nu \colon C\times \Hom(C,B)\to\cC$ from the universal framed
curve, and as the map $dg_c+r_c \colon T_c C \oplus H^0(C, g^*\shT_B)\to T_{g(c)}B$ is the
derivative of the universal framed map $G\circ\nu$ \cite[II.3.4]{kollarrational}, it
follows that the induced map $T_c \, \cC\to T_{g(c)}B$ is the derivative of
$G$ at the point $c \in \cC$.  Thus, again since $g^*\shT_B$ is globally generated,
the evaluation map $G \colon \cC \to B$ is smooth at every point of $\cC$.

\begin{remark}
	The rational curves in our family are free, and so a generic curve $g \colon C \to
	B$ is an immersion \cite[II.3.14]{kollarrational}, but not necessarily an
	embedding. Of course, we will eventually show that $B \cong \P^n$ and that $M$ is
	the space of lines in $\P^n$, but one has to be careful not to make any
	unjustified assumptions during the proof.
\end{remark}

Let us summarize the conclusions that we have drawn from \Cref{prop:semipositive} and
the deformation theory of rational curves.

\begin{lemma} \label{lem:smoothness}
The moduli space $M$ is smooth of dimension $n+d-2$, the universal curve $\pi \colon
\cC \to M$ is a $\P^1$-bundle over $M$, and the evaluation morphism $G \colon \cC \to B$ is
smooth of relative dimension $d-1$.
\end{lemma}

We denote by $G^{\circ} \colon \cC^{\circ} \to B^{\circ}$ the base change
to $B^{\circ}$, and by $\pi^{\circ} \colon \cC^{\circ} \to M$ the projection to the
moduli space. By our choice of $d$, any rational curve of anticanonical
degree $d$ that meets $B^{\circ}$ must be irreducible and therefore cannot
degenerate; this implies that $G^\circ$ is proper.

Now the general idea is the following. Suppose that $d < n$, which means that the
variation of Hodge structure $V$ on $B^{\circ}$ has a nontrivial fixed part on a
general curve in our family. We are going to argue that there is a finite covering
space of $B^{\circ}$, constructed from the Stein factorization of $G^{\circ} \colon
\cC^{\circ} \to B^{\circ}$, on which the pullback of $V$ decomposes in a nontrivial
way. As a first step, we are going to construct such a decomposition on the universal
curve $\cC^{\circ}$ itself.

Consider the pullback $V_{\cC} = (G^{\circ})^* V$ of the variation of Hodge structure to
$\cC^{\circ}$, and its direct image $\pi_*^{\circ} V_{\cC}$. By general
theory, this is a constructible sheaf on $M$. Let $T \subseteq M$ be the (dense)
Zariski open subset where this constructible sheaf is a local system (and therefore a
variation of Hodge structure) of rank $2(n-d)$; this includes all those curves $g
\colon C \to B$ on which $g^* \shT_B$ decomposes as in
\eqref{eq:generic-decomposition}. (In fact, we will show later that this is true for
every curve, and so actually $T = M$; but this does not matter for the time being.)
By the semi-simplicity of polarized variations of Hodge structure, we get a canonical
splitting of rational variations of Hodge
structure
\begin{equation}\label{upstairs splitting}
	V_{\cC} \cong \Vfix_\cC\oplus W_\cC
\end{equation}
where $\Vfix_\cC$ is the unique subvariation of Hodge structure that agrees with the
pullback $(\pi^{\circ})^* (\pi^{\circ})_* V_{\cC}$ on the open subset
$\cC_T = (\pi^{\circ})^{-1}(T)$; this exists because the fundamental group of this open
subset surjects onto that of the complex manifold $\cC^{\circ}$. Note that
$\Vfix_{\cC}$ is itself a variation of Hodge structure of weight $1$ and rank $2(n-d)$ that
represents the ``universal'' fixed part of the restriction to the fibers. 
In fact, the decomposition in \eqref{upstairs splitting} descends to a
finite covering of $B^{\circ}$, as we now explain.

The fiber $G^{-1}(b)$ over a point $b \in B$, or rather its image in the moduli space
$M$, parametrizes all those rational curves in our family that pass
through $b$. In the case of projective space, the space of lines through a given
point is irreducible, but this might (a priori) not be the case here. We therefore
consider the Stein factorization 
\[
	\begin{tikzcd}
		\cC^{\circ} \arrow[bend left=35]{rr}{G^{\circ}} \rar{s^{\circ}} & Z^{\circ}
		\rar{p^{\circ}} & B 
	\end{tikzcd}
\]
of the proper smooth morphism $G^{\circ} \colon \cC^{\circ} \to B^{\circ}$; here
$s^{\circ}$ is smooth with connected fibers (of constant dimension $d-1$), and
$p^{\circ}$ is finite \'etale. Let
$s \colon Z \to B$ be the unique extension of this finite covering space to a finite
morphism from a normal projective variety $Z$. Because the universal curve $\cC$ is
smooth by \Cref{lem:smoothness}, we get the following commutative diagram:
\begin{equation} \label{eq:Stein-factorization}
	\begin{tikzcd}
		\cC \dar{\pi} \arrow[bend left=35]{rr}{G} \rar{s} & Z \rar{p} & B \\
		M
	\end{tikzcd}
\end{equation}

\begin{remark}
	The fact that every rational curve $g \colon C \to B$ in our family factors
	through the finite morphism $p \colon Z \to B$ means that the fundamental group of
	$C^{\circ}$ does not generate the fundamental group of $B^{\circ}$ if the degree of $p$ is positive, unlike in the
	case of lines on projective space.
\end{remark}

We again summarize all the relevant facts about the commutative diagram.

\begin{lemma}\label{lem:splitting}
	With the notation introduced above, the following is true:
\begin{enumerate}
	\item The variety $Z$ is smooth at every point of the open subset $s(\cC)$, and
		the map $p \colon Z \to B$ is \'etale there.
\item There is a splitting of variations of Hodge structure
\begin{equation}\label{downstairs splitting}
	V_Z:=(p^\circ)^*V = \Vfix_Z \oplus W_Z
\end{equation}
of weight $1$ that pulls back to \eqref{upstairs splitting}.
\item The decomposition in \eqref{downstairs splitting} induces a decomposition of
	the tangent bundle 
	\[
		\shT_{s(\cC)} \cong \shTfix \oplus \shTpos
	\]
	on $s(\cC)$, and the two summands define foliations of rank $n-d$ respectively $d$.
	\item All fibers of $\pi \colon \cC \to M$ are tangent to the foliation $\shTpos$.
\end{enumerate}
\end{lemma}

\begin{proof}
	(1) follows from the fact that $G \colon \cC \to B$ is smooth. For (2), observe
	that because $\cC^\circ$ and $Z^{\circ}$ are smooth, the composition
	$\pi_1(\cC^\circ_T)\to \pi_1(\cC^\circ) \to \pi_1(Z^\circ)$ is surjective; this
	implies the claim. The restriction of $p \colon Z \to B$ to the open subset
	$s(\cC)$ is \'etale, and so
	\[
		\shT_{s(\cC)} \cong (p^* \shT_B) \vert_{s(\cC)}.
	\]
	Because $\shT_B$ is the Hodge module extension of $F^1 \shV$, it follows that
	$\shT_{s(\cC)}$ is the Hodge module extension of $F^1 \shV_Z$. The decomposition
	of the variation of Hodge structure therefore induces a decomposition of $\shT_Z$
	into a vector bundle of rank $n-d$ and a vector bundle of rank $d$. The two
	summands are foliations on $s(\cC)$ by \Cref{lem:foliation}.
	It remains to prove (4). By construction, the pullback of $\shTfix$ to a general curve
	$g \colon C \to B$ in our moduli space is a trivial bundle of rank $n-d$. Since
	$g^* \shT_B$ is always semipositive, it follows (by the deformation invariance of
	the degree) that the pullback of $\shTfix$ is trivial for
	\emph{every} curve in $M$. For degree reasons, $\shT_C \cong \shO(2)$ must
	therefore be contained in the pullback of the other summand, and so all curves in
	the universal family are indeed tangent to the foliation $\shTpos$.
\end{proof}

\section{Connected components of the family of rational curves}\label{sect connected comp}

We are going to need two additional results about the connected components of the
family of rational curves through a given point. We keep the notation introduced in
\eqref{eq:Stein-factorization}. At least generically, the fibers of $p \colon Z \to
B$ represent the different connected components of the family of rational curves
through a point. (This interpretation is slightly handicapped by the fact that
a generic rational curve in our family can a priori have nodes.) Indeed, a point $z
\in Z^{\circ}$ can be thought of as the point $p(z) \in B^{\circ}$, together with the
connected component $s^{-1}(z)$ of the fiber $G^{-1}\bigl( p(z) \bigr)$; of course,
the connected component of the family of curves through $b$ is really the image $M_z =
\pi(s^{-1}(z))$ in the moduli space $M$. For any $z\in Z$, we denote by
\begin{equation} \label{eq:Pz}
	\Pt_z:=s \bigl( \pi^{-1}(M_z) \bigr) \subseteq Z
\end{equation} 
the locus in $Z$ swept out by the curves in the given component.\footnote{This
notation is different from Hwang's notation in \cite{Hwang:base}.} 
By construction, $z \in P_z$. The following result, which is mostly due to Araujo
\cite{Araujo}, shows that $P_z$ is generically a projective space.

\begin{lemma} \label{lem:Pz}
	For any $z \in Z^{\circ}$, the following is true:
	\begin{enumerate}
	\item The subvariety $P_z \subseteq Z$ is isomorphic to $\P^d$, and 
		its normal bundle in $Z$ is a trivial bundle of rank $n-d$.
	\item At any point $z' \in P_z \cap Z^{\circ}$, the isomorphism
		\[
			T_{z'} Z \cong T_{p(z')} B \cong F^1 \shV \vert_{p(z')} 
			\cong F^1 \shV_Z \vert_{z'}
		\]
		maps the tangent space $T_{z'} P_z$ isomorphically to the fiber $F^1 \shW_Z
		\vert_{z'}$.
	\item Two subvarieties $P_z$ and $P_{z'}$ are either equal or disjoint.
\end{enumerate}
\end{lemma}

\begin{proof}
	Since $G^{\circ} \colon \cC^{\circ} \to B^{\circ}$ is proper, it is clear that
	$P_z$ is projective. Let us first show that $\dim P_z = d$. Since $\pi \colon \cC
	\to M$ is a
	$\P^1$-bundle, the subvariety $s^{-1}(z)$ must be finite over its image $M_z$ in
	$M$ (because otherwise $G$ would be constant along the curve in question).
	Therefore $\dim M_z = d-1$ and $\dim \pi^{-1}(M_z) = d$. Now observe that the
	projection from $\pi^{-1}(M_z)$ to $P_z$ is finite away from $p^{-1}(z)$. Indeed,
	if the fiber over a point $z' \in P_z$ with $p(z')
	\neq p(z)$ had positive dimension, we would get a nontrivial family of rational
	curves through the two fixed points $p(z)$ and $p(z')$, and by by Mori's Bend and
	Break technique, this would produce a rational curve through $p(z)$ of smaller
	anticanonical degree, contradicting our initial choice.

	By construction, we have $P_z \subseteq s(\cC)$. From \Cref{lem:splitting}, we
	get a decomposition
	\[
		\shT_{s(\cC)} \cong \shTfix \oplus \shTpos,
	\]
	and the summand $\shTpos$ is a foliation of rank $d$. Since $P_z$ is a union of
	curves that are tangent to this foliation, it follows that the tangent space to $P_z$
	at each point is contained in the fiber of $\shTpos$. For dimension reasons, this
	means that $P_z$ is smooth and $T_{z'} P_z = \shTpos \big\vert_{z'}$ for
	every $z' \in P_z$, which gives (2). We also get
	\[
		\shT_{P_z} \cong \shTpos \big\vert_{P_z} \quad \text{and} \quad
		N_{P_z|Z} \cong \shTfix \big\vert_{P_z}.
	\]
	Moreover, any two leaves $P_z$ and $P_{z'}$ are either equal or disjoint, proving
	(3).

	Now let $g \colon C \to B$ be any curve in the subset $M_z \subseteq M$, and
	denote by $i \colon C \to P_z$ the induced morphism to $P_z$. Because we already
	know that the restriction of $\shTfix$ to every curve in $M$ is trivial, we get
	\[
		\deg i^* \shT_{P_z} = \deg i^* \shT_{s(\cC)} = \deg g^*
		\shT_B = d+1.
	\]
	This means that there is a rational curve of anticanonical degree $d+1$ through
	any two points of $P_z$ (namely the rational curves that sweep out the subvariety
	$P_{z'}$ for any $z' \in P_z$), and by our choice of $d$, these curves can
	obviously not degenerate. We can therefore apply \cite[Thm.~0.1]{cmsb} and
	conclude that $P_z \cong \P^d$ and that the curves in the subset $M_z \subseteq M$
	are exactly the lines through the point $z \in P_z$. Because $p \colon Z \to B$ is
	\'etale on the open subset $s(\cC)$, it also follows that
	\[
		g^{\ast} \shT_B \cong i^* \shT_{P_z} \oplus \shO^{\oplus(n-d)}
		\cong \shO(2) \oplus \shO^{\oplus(d-1)} \oplus \shO^{\oplus(n-d)},
	\]
	and so the pullback of the tangent bundle of $B$ actually has the same splitting
	behavior as in \eqref{eq:generic-decomposition} on \emph{every} curve in the
	universal family. 

	Finally, the normal bundle $N_{P_z|Z}$ is a vector bundle of rank $n-d$ on $P_z
	\cong \P^d$ whose restriction to every line is trivial; by general facts about
	vector bundles on projective space, it follows that $N_{P_z|Z}$ is a trivial
	bundle of rank $n-d$, and (1) is proved.
\end{proof}

\begin{remark}
	One can use the results above to show that the open subset $s(\cC) \subseteq Z$ is
	actually $\P^d$-bundle over a smooth variety of dimension $n-d$, whose fibers are the
	projective spaces $P_z$. The key point is that relevant Hilbert scheme of $Z$ is
	smooth of dimension $n-d$ at the point corresponding to $P_z$, due to the fact
	that the normal bundle of $P_z$ in $Z$ is trivial. For details, see \cite[Lem.~3.3]{Araujo}.
\end{remark}

The second result is a proposition proved by Hwang in his earlier paper
\cite{Hwang:pre-base}. As it plays an important role in our
argument, we include the proof. First a lemma.

\begin{lem} \label{lem:Hwang-intersection}
	Let $z_1,z_2 \in Z^{\circ}$ be two points. If $p(P_{z_1}) = p(P_{z_2})$, then
	$P_{z_1} = P_{z_2}$.
\end{lem}

\begin{proof}
	Suppose that $p(P_{z_1}) = p(P_{z_2})$, and denote by $P$ the
	normalization of this subvariety of $B$. Since $P_{z_j} \subseteq s(\cC)$, the induced
	morphism $P_{z_j} \to P$ is a finite immersion, and therefore a finite covering space.
	But $P_{z_j} \cong \P^d$, and therefore $P \cong \P^d$; this is simply connected, and
	so $P_{z_j} \to P$ is an isomorphism. By composing one isomorphism with the
	inverse of the other, we get an isomorphism between $P_{z_1}$ and 
	$P_{z_2}$. Under the isomorphism $P_{z_j} \cong \P^d$, any line in $\P^d$ determines a
	rational curve in $B$ of anticanonical degree $d+1$. Since every automorphism of
	$\P^d$ takes lines to lines, it follows that $\pi \bigl( s^{-1}(P_{z_1}) \bigr) = \pi
	\bigl( s^{-1}(P_{z_2}) \bigr)$, and therefore $P_{z_1} = P_{z_2}$.
\end{proof}

The following result is a slight generalization of \cite[Prop.~2.2]{Hwang:pre-base}.

\begin{prop} \label{prop:Hwang-tangents}
	Let $z_1,z_2 \in Z^{\circ}$ be two points with $p(z_1) = p(z_2) = b$. If the 
	subspaces
	\[
		F^1 \shW_Z \vert_{z_1} \subseteq F^1 \shV_Z \vert_{z_1} \cong F^1 \shV \vert_b
		\quad \text{and} \quad
		F^1 \shW_Z \vert_{z_2} \subseteq F^1 \shV_Z \vert_{z_2} \cong F^1 \shV \vert_b
	\]
	have a nontrivial intersection in $F^1 \shV \vert_b$, then $P_{z_1} = P_{z_2}$.
\end{prop}

\begin{proof}
	Since $p \colon Z \to B$ is \'etale at the point $z_1$, we can find an open
	neighborhood $U \subseteq B$ of the point $b$, and and open neighborhood $U_1
	\subseteq Z$ of the point $z_1$, such that $p \vert_{U_1} \colon U_1 \to U$ is
	biholomorphic. The pullback of $W_Z \vert_{U_1}$ along the inverse of this
	biholomorphic mapping therefore defines a subvariation of Hodge structure $W_1$ of
	$V \vert_U$, whose rank is $2d$. After applying the same construction to the point
	$z_2$, we obtain another subvariation of Hodge structure $W_2$ of $V \vert_U$,
	also of rank $2d$. Since the two Hodge bundles $F^1 \shW_1$ and $F^1 \shW_2$ have
	a nontrivial intersection at the point $b$, the intersection $W = W_1 \cap W_2$
	must be a nontrivial subvariation. According to \Cref{lem:foliation}, the
	intersection $F^1 \shW_1 \cap F^1 \shW_2 = F^1 \shW$ is therefore a foliation of
	positive rank on $U$. The leaf of the foliation through the point $b$ is contained
	in both $p(P_{z_1})$ and $p(P_{z_2})$, and so the intersection $p(P_{z_1}) \cap
	p(P_{z_2})$ must have positive dimension.

	As $P_z \to p(P_z)$ is always birational, we can therefore find a smooth projective
	curve $Y$ and two morphisms $f_1 \colon Y \to P_{z_1}$ and $f_2 \colon Y \to
	P_{z_2}$ such that $p \circ f_1 = p \circ f_2$. We can clearly arrange that the
	image of $Y$ passes through the point $b$. The following commutative diagram
	shows the relevant morphisms:
	\[
		\begin{tikzcd}
			Y \rar{f_1} \dar{f_2} & P_{z_1} \dar{p} \\
			P_{z_2} \rar{p} & B
		\end{tikzcd}
	\]
	From \Cref{lem:Pz}, we know that $p^* \shT_B \vert_{P_z} \cong \shT_{P_z} \oplus
	\shO^{\oplus(n-d)}$, and that the tangent bundle to $P_z \cong \P^d$ is the
	restriction of the bundle $\shTpos$. Since the tangent bundle of projective space
	is ample, the identity
	\[
		f_1^* \shT_{P_{z_1}} \oplus \shO^{\oplus(n-d)} =
		f_2^* \shT_{P_{z_2}} \oplus \shO^{\oplus(n-d)}
	\]
	inside the pullback of $\shT_B$ clearly implies that $f_1^* \shT_{P_{z_1}}$ and
	$f_2^* \shT_{P_{z_2}}$ are the same subbundle. But this means that $F^1 \shW_Z
	\vert_{z_1}$ and $F^1 \shW_Z \vert_{z_2}$ are equal, as subspaces of $F^1 \shV
	\vert_b$, from which it follows as above that $p(P_{z_1}) \cap p(P_{z_2})$ has
	dimension $d$, and hence that $p(P_{z_1}) = p(P_{z_2})$. We now reach the desired
	conclusion by applying \Cref{lem:Hwang-intersection}.
\end{proof}

\section{Controlling the ramification of the finite covering} \label{sect cover}

After the results in the four preceding sections, the proof of Hwang's theorem is
reduced to showing that the splitting in \Cref{lem:splitting} descends to
$B^{\circ}$ (where it contradicts the irreducibility of $V$). Our proof of this fact
is a version of \cite[\S 5]{Hwang:base}, but becomes simpler because
everything is controlled by variations of Hodge structure. The crucial step is to
control the ramification of the finite morphism $p \colon Z \to B$ via the following
result.

\begin{prop}\label{prop:Hwang}
	Let $g \colon C \to B$ be the rational curve corresponding to a general point in
	the moduli space $M$. Then the base change of the finite covering space
	$p^\circ:Z^\circ\to B^\circ$ along $g^\circ \colon C^{\circ} \to B^{\circ}$ is a
	disjoint union of copies of $C^\circ$.
\end{prop}

Let $g \colon C \to B$ be a generic rational curve (in the universal family $\cC$).
Such a curve intersects the discriminant locus $D$ transversely (because all the
rational curves under consideration are free), and as $Z$ is normal, the base
change 
\[
	\begin{tikzcd}
		C \times_B Z \rar \dar & Z \dar{p} \\
		C \rar{g} & B
	\end{tikzcd}
\]
is a finite union of smooth projective curves, each with a finite morphism to $C$.
Since $G \colon \cC \to B$ factors through $p \colon Z \to B$ by construction,
exactly one of these curves is (the image of) $C \subseteq \cC$. Let $p \colon R \to
C$ be one of
the other components. If $\deg(R/C) = 1$, then $p$ is an isomorphism, and so
$R^{\circ}$ is isomorphic to $C^{\circ}$. In order to prove \Cref{prop:Hwang}, we may
therefore assume from now on that $\deg(R/C) \geq 2$; after several steps, we will
see that this assumption leads to a contradiction.  The idea is the same as in Hwang's paper:  to first reduce to the $d=1$ case (so the curves in $M$ don't deform with a fixed point), and then to show there is a deformation of $C$ for which $R\to C$ deforms trivially but for which constant sections deform along curves in the family $M$. 
 The condition $\deg(R/C)\geq 2$ then contradicts \Cref{prop:Hwang-tangents}.

\begin{lem} \label{lem:constant}
	If $\deg(R/C) \geq 2$, then the pullback of the variation of Hodge structure
	$V$ to $R^{\circ}$ is constant.
\end{lem}

\begin{proof}
	We use the superscript ``$\circ$'' to denote the base change to the open subset
	$B^{\circ}$. We have the following commutative diagram:
	\[
		\begin{tikzcd}
			R^{\circ} \dar{p^{\circ}} \rar & Z^{\circ} \dar{p^{\circ}} \\
			C^{\circ} \rar{g^{\circ}} & B^{\circ}
		\end{tikzcd}
	\]
  We obtain two decompositions as variations of Hodge structure of
  $V_R:=V_Z \vert_{R^{\circ}}$. One decomposition $V_R \cong \Vfix_R \oplus W_R$
  comes from restricting \eqref{downstairs splitting} to $R^{\circ}$; the other comes
  from pulling back $V_C = (g^{\circ})^* V \cong \Vfix_C \oplus W_C$ along the map $p^{\circ}
  \colon R^{\circ} \to C^{\circ}$. Note that both $W_C$ and $W_R$ are variations of
  Hodge structure of rank $2d$, where $d$ is the integer from \eqref{eq:degree}.  We
  now play the two decompositions off against each other.

  Consider the intersection $(g^{\circ})^* W_C \cap W_R$, which is again a
  subvariation of Hodge structure of $V_R$. If it is nontrivial, then the
  intersection of the two Hodge bundles $(g^{\circ})^* F^1 \shW_C$ and $F^1 \shW_R$
  is a vector bundle of positive rank on all of $R^{\circ}$. According to
  \Cref{prop:Hwang-tangents}, this gives us $P_{\gt(y)} = P$ for every $y \in
  R^{\circ}$, where $P$ is the projective space containing our original curve. But
  this is clearly impossible. The conclusion is that $(g^{\circ})^* W_C$
  and $W_R$ intersect trivially. This makes the composition
  \[
	  W_R \to V_R \to (p^{\circ})^* \Vfix_C
  \]
  of the inclusion and the projection injective. As $\Vfix_C$ is a constant
  variation of Hodge structure on $C^{\circ}$, it follows that $W_R$ is constant,
  too.

  Now we claim that the intersection of the subvariations $W_R$ and $(p^{\circ})^*
  \Vfix_C$ is also trivial. Otherwise, the Hodge bundle $F^1 \shW_R$ contains a
  subbundle of positive rank that is the pullback of a trivial bundle from
  $C^{\circ}$, which means that the two subspaces $F^1 \shW_R \vert_y$
  and $F^1 \shW_R \vert_{y'}$ intersect inside $F^1 \shV_{p(y)}$, for any
  two points $y,y' \in R^{\circ}$ such that $p(y) = p(y')$. By \Cref{prop:Hwang-tangents},
  it follows that $P_{\gt(y)} = P_{\gt(y')}$, and as $\deg(R/C) \geq 2$, this is clearly
  impossible.  The conclusion is that, for dimension reasons, 
  \[
	  V_R \cong W_R \oplus (p^{\circ})^* \Vfix_C
  \]
  is the sum of two constant variations of Hodge structure on $R$.
\end{proof}
\begin{remark}\label{rmk:nonisotrivial is done}\Cref{lem:constant} in particular implies there is a covering family of curves on which $V$ is isotrivial, which contradicts the main result of \cite{Bakker:Matsushita} (and therefore proves \Cref{prop:Hwang}) unless the Lagrangian fibration is isotrivial.  In the isotrivial case the foliations all come from a flat structure on $\shT_{B^\circ}$, and are therefore easier to think about, but we do not in fact use this reduction.
    
\end{remark}
The next step is to show that $d=1$; this means that, in our family, the are only
finitely many rational curves through any given point of $B^{\circ}$.

\begin{lem}
	In the above situation, we must have $d=1$.
\end{lem}

\begin{proof}
		Because $p \colon R \to C$ is \'etale over $C^{\circ}$, it must ramify over at
	least one point $x \in C$ such that $g(x) \in D$. Being generic, $g \colon C \to B$
	intersects the discriminant locus $D$ transversely, and so $g(x)$ is a smooth
	point of $D$. The pullback of $V$ to $R^{\circ}$ has trivial monodromy (by
	\Cref{lem:constant}), hence the Hodge structures do not vary from point to point;
	as $C^{\circ}$ is quasi-projective, this implies that the Hodge bundle $F^1 \shV_C$
	is a flat subbundle of $\shV_C$. It also follows that the local monodromy of $V_C$
	around the point $x$ has finite order; consequently, the local monodromy of
	$V$ near $g(x)$ has finite order as well. According to \Cref{lem:finite-monodromy},
	the monodromy transformation therefore has exactly two nontrivial eigenvalues, and
	because the Hodge bundle $F^1 \shV_C$ is a flat subbundle of $\shV_C$, it contains
	(in a punctured neighborhood of the point $x$) a subbundle of rank $n-1$ that is
	preserved by the monodromy transformation around $x$. Now
	if $d \geq 2$, then it follows as above that at any two points $y,y' \in
	R^{\circ}$ with $p(y) = p(y')$ sufficiently close to $x$, the two subspaces $F^1
	\shW_R \vert_y$ and $F^1 \shW_R \vert_{y'}$ intersect nontrivially. As this is
	impossible (by \Cref{prop:Hwang-tangents}), we conclude that $d=1$.
\end{proof}

This reduces the problem to the case $d=1$. The smooth morphism $G \colon \cC \to B$
now has finite fibers, and so $s \colon \cC \to Z$ is an open embedding. We may
therefore identify the universal curve $\cC$ with an open subset of $Z$ for the
remainder of the argument. We have $\dim \cC = n$ and $\dim M = n-1$; because there
is nothing to prove when $n=1$, we shall assume from now on that $n \geq 2$.

Recall from \Cref{lem:smoothness} that $p \colon Z \to B$
is \'etale at every point of $\cC$. According to \Cref{lem:splitting}, we have a
decomposition of the tangent bundle
\begin{equation} \label{eq:decomposition-cC}
	\shT_{\cC} \cong \shCfix \oplus \shCpos
\end{equation}
into two foliations of rank $n-1$ respectively $1$, induced by the decomposition
\[
	V_Z = \Vfix_Z \oplus W_Z
\]
of the variation of Hodge structure $V_Z = (p^{\circ})^* V$ on $Z^{\circ}$. From the
construction in \eqref{upstairs splitting}, it is clear that $\shCfix$ is the
pullback of a vector bundle of rank $n-1$ from the moduli space $M$, whereas
$\shCpos$ is tangent to the fibers of $\pi \colon \cC \to M$. By comparing
\eqref{eq:decomposition-cC} with the short exact sequence for the relative tangent
bundle
\[
	0 \to \shT_{\cC/M} \to \shT_{\cC} \to \pi^* \shT_M \to 0,
\]
we deduce that $\shCfix \cong \pi^* \shT_M$ and $\shCpos \cong \shT_{\cC/M}$.

Denote by $\gt \colon R \to Z$ the morphism induced by $p \colon R \to B$, as in the
following commutative diagram:
\[
	\begin{tikzcd}
		R \rar{\gt} \dar{p} & Z \dar{p} \\
		C \rar{g} & B
	\end{tikzcd}
\]
The intersection $R \cap \gt^{-1}(\cC)$ is a dense Zariski open subset of the curve
$R$ that contains $R^{\circ}$.  Under the projection $\pi \colon \cC \to
M$, which is proper, the image $\pi(\gt(R) \cap \cC)$ is a quasi-projective curve
inside the moduli space $M$.

\begin{lem} \label{lem:unbranched}
	The morphism $\pi \circ \gt \colon R \cap \gt^{-1}(\cC) \to \pi(\gt(R) \cap \cC)$
	is an immersion.
\end{lem}

\begin{proof}
	Suppose to the contrary that $y \in R \cap \gt^{-1}(\cC)$ is a branch point. Set
	$z = \gt(y) \in \gt(R) \cap \cC$. This
	means that $\gt(R)$ is tangent to the fiber $P_z = \pi^{-1} \bigl( \pi(z) \bigr)$ at
	the point $z$. Since $p \colon Z \to B$ is \'etale there, we have $T_z Z \cong
	T_{p(z)} B$. Under this isomorphism, the tangent space $T_y R$ maps isomorphically
	to $T_{p(z)} C$, and the tangent space $T_z P_z$ is isomorphic to $\shCpos
	\vert_z$. By \Cref{lem:constant}, the variation of Hodge structure $(p^{\circ})^*
	V_C$ on $R^{\circ}$ is constant. It therefore extends uniquely to a constant
	variation of Hodge structure $V_R$ on the entire curve $R$, and both $W_Z
	\vert_{R^{\circ}}$ and $(p^{\circ})^* W_C$ extend to constant subvariations of
	rank $2$. The condition that $T_z Z$ and $T_y R$ have the same image inside
	$T_{p(z)} B$ means that these two subvariations are equal at the point $y$. Being
	constant, they must then be equal everywhere, and as usual, this is in
	contradiction with \Cref{prop:Hwang-tangents}.
\end{proof}

We now investigate what happens when we deform the rational curve $g \colon C \to B$.
Because $\pi \colon \cC \to M$ is a
$\P^1$-bundle, we can choose a small open neighborhood $M_0$ of the point $g \colon
C \to B$ in the moduli space $M$, say with $M_0$ biholomorphic to an open ball,
such that $\pi^{-1}(M_0) \cong M_0 \times \P^1$. The original rational curve is now
a morphism $g \colon \P^1 \to B$, and the nearby subvarieties $P_z$ are the different
copies of $\P^1$. The following diagram shows the relevant morphisms:
\[
	\begin{tikzcd}
		M_0 \times \P^1 \dar{\pi} \rar{G} & B \\
		M_0	
	\end{tikzcd}
\]
By restricting the decomposition in \Cref{lem:splitting} to the open subset $M_0 \times
\P^1$, we obtain a decomposition of the tangent bundle
\[
	\shT_{M_0 \times \P^1} \cong \shUPfix \oplus \shUPpos
\]
in which the two summands are foliations of rank $n-1$ respectively $1$. By
construction, the fibers of $\pi$ are leaves of $\shUPpos$, and the restriction of
$\shUPfix$ to the fibers is trivial. 

Since the initial curve $g \colon \P^1 \to B$ intersects $D$ transversely,
the inverse image $G^{-1}(D)$ is a union of finitely many sections $(\id \times
h_j)(M_0)$, where $h_1, \dotsc, h_m \colon M_0 \to \P^1$ are holomorphic mappings.
The next result says that each of these sections is actually a leaf of the foliation
$\shUPfix$.

\begin{lem} \label{lem:discriminant-tangent}
	The foliation $\shUPfix$ is tangent to the divisor $G^{-1}(D)$.
\end{lem}

\begin{proof}
	The argument is very similar to the proof of \Cref{lem:finite-monodromy}. The
	holomorphic mapping from $M_0 \times \P^1 \to Z$ is an open embedding, and by 
	\Cref{lem:splitting}, the tangent bundle $\shT_{M_0 \times \P^1}$ is therefore the
	canonical extension of the pullback of the Hodge bundle $F^1 \shV_Z$. Since $V_Z =
	\Vfix_Z \oplus W_Z$, we can apply the construction in \Cref{lem:finite-monodromy}
	in a neighborhood of any component of $G^{-1}(D)$, but choosing the sections
	$v_1$ and $v_{n+1}$ in the canonical extension of $W_Z$, and the remaining
	sections $v_2, \dotsc, v_n$ and $v_{n+2}, \dotsc, v_{2n}$ in the canonical
	extension of $\Vfix_Z$. Then $s_1$ is a local frame for the canonical extension of $F^1
	\shW_Z$, and $s_2, \dotsc, s_n$ are a local frame for the canonical extension of
	$F^1 \shVfix_Z$, and therefore for the foliation $\shUPfix$. The same argument as
	before shows that $s_2, \dotsc, s_n$ are tangent to the given component of
	$G^{-1}(D)$, and this proves the claim.
\end{proof}

The tangent bundle of $M_0 \times \P^1$ is globally generated (because $\shT_{\P^1}$
is globally generated and $M_0$ is a Stein manifold); the same thing is therefore
true for the foliation $\shUPfix$.
After composing $G$ with an automorphism of $M_0 \times \P^1$, we can arrange that
the three standard sections $M_0 \times \{0\}$, $M_0 \times \{1\}$, and
$M_0 \times \{\infty\}$ are leaves of the foliation $\shUPfix$. From the isomorphism
\[
	H^0 \bigl( M_0 \times \P^1, \shT_{M_0 \times \P^1} \bigr)
	\cong H^0(M_0, \shT_{M_0}) \oplus H^0(M_0, \shO_{M_0}) 
	\otimes H^0 \bigl( \P^1, \shT_{\P^1} \bigr)
\]
we then conclude that \emph{every} leaf of $\shUPfix$ must be of the form $M_0 \times \{x\}$
for a point $x \in \P^1$. In particular, every component of the divisor $G^{-1}(D)$
must be of this form, which means that the holomorphic mappings $h_1, \dotsc, h_m
\colon M_0 \to \P^1$ are constant. 

From this fact, we can easily deduce that the finite covering $p \colon R \to \P^1$
does not change as we deform $g \colon \P^1 \to B$. Recall that $R$ is a smooth
projective curve, and that $\gt \colon R \to Z$ is the induced morphism.
After shrinking $M_0$, if necessary, the base change
\[
	\begin{tikzcd}
		(M_0 \times \P^1) \times_B Z \rar \dar & Z \dar{p} \\
		M_0 \times \P^1 \rar{G} & B
	\end{tikzcd}
\]
is also smooth; let $\cR$ be the unique component containing $R$. We have
just shown that $\cR^{\circ} \to M_0 \times (\P^1)^{\circ}$ is a finite covering
space; because finite covering spaces of a fixed quasi-projective curve do not deform
in a nontrivial way, we get $\cR \cong M_0 \times R$. The following commutative
diagram shows the relevant morphisms:
\[
	\begin{tikzcd}
		& \cC \dar[hook] \rar{\pi} & M \\
		M_0 \times R \rar{\Gt} \dar{\id \times p} & Z \dar{p} \\
		M_0 \times \P^1 \rar{G} & B
	\end{tikzcd}
\]
According to \Cref{lem:unbranched}, the morphism
\[
	\pi \circ \gt \colon R \cap \gt^{-1}(\cC) \to \pi \bigl( \gt(R) \cap \cC \bigr)
\]
is unramified. Because unramified morphisms between two fixed quasi-projective curves
also do not deform in a nontrivial way, the following lemma will very quickly lead to
the desired contradiction.

\begin{lem} \label{lem:contradiction}
	There is a $1$-dimensional family of deformations of the initial rational curve $g
	\colon \P^1 \to B$ for which the curves $\pi(\gt(R) \cap \cC) \subseteq M$ stay
	the same.
\end{lem}

\begin{proof}
	Consider the quasi-projective curve $\pi(\gt(R) \cap \cC)$ inside the
	moduli space $M$. Its preimage under $\pi$ is a quasi-projective surface in $\cC$.
	Since $p \colon \cC \to B$ is smooth of relative dimension $0$, it follows that
	\begin{equation} \label{eq:surface}
		S = p \Bigl( \pi^{-1} \bigl( \pi(\gt(R) \cap \cC) \bigr) \Bigr) \subseteq B
	\end{equation}
	is a quasi-projective surface in $B$ that contains the intersection $g(\P^1) \cap
	B^{\circ}$. A priori, $S$ could be singular along this curve, but we will show in a
	moment that it is actually smooth in a neighborhood of $g(\P^1)
	\cap B^{\circ}$. We can think of $S$ as being swept out by those rational curves
	in our family that intersect $R$.

	Let us first prove that $S$ is smooth at every point $b_0 \in g(\P^1) \cap
	B^{\circ}$. Let $z_0 \in \cC^{\circ}$ be a point on the curve $g \colon \P^1 \to B$
	such that $p(z_0) = b_0$, and enumerate the points in the fiber $p^{-1}(b_0) \cap
	R^{\circ}$ as $y_1, \dotsc, y_r$, where $r = \deg(R/\P^1)$; set $z_j = \gt(y_j)$.
	Choose an open neighborhood $U \subseteq B$ of the point
	$b_0$, and an open neighborhood $U_j \subseteq \cC$ of
	each point $z_j$, such that $p \vert_{U_j} \colon U_j \to U$ is a biholomorphism.
	Pulling back the $r+1$ variations of Hodge structure $W_Z \vert_{U_j}$, we get
	$r+1$ subvariations $W_0, \dotsc, W_r \subseteq V \vert_U$.  From the proof of
	\Cref{lem:constant}, we know that the subvariation $W_R + (p^{\circ})^* W_C$ of
	$V_R$ has rank $4d=4$, and that its fibers at any two points $y,y' \in R^{\circ}$
	with $p(y) = p(y')$ map to the same subspace of $T_{p(y)} B$. This implies that $W
	= W_0 + W_1 + \dotsb + W_r$ has rank
	$4$ at the point $b_0$, and therefore on all of $U$. Consequently, the Hodge
	bundle $F^1 \shW$ defines a foliation of rank $2$ on the open set $U$, and this
	foliation is the sum of the two foliations $F^1 \shW_0$ and $F^1 \shW_j$ of rank
	$1$, for every $j=1, \dotsc, r$.

	From \eqref{eq:surface}, it is easy to see that all those curves involved in the
	construction of $S$ that intersect the open subset $U$ must be tangent to the
	foliation $F^1 \shW$. For dimension reasons, this implies that $S \cap U$ is the
	leaf of the foliation $F^1 \shW$ through the point $b_0$, and therefore smooth.

	Now let $C_0$ denote the branch of the curve $g \colon \P^1 \to B$ corresponding to
	the point $z_0$; to be precise, we have $C_0 = p(U_0 \cap P_{z_0})$. Similarly,
	let $C_j = p(U_j \cap P_{z_j})$ be the branches of the other $r$ rational curves.
	By construction, each $C_j$ is the leaf of the foliation $F^1 \shW_j$ through
	$b_0$. According to \Cref{lem:integrable-sum}, the surface germ $S \cap U$ is
	therefore swept out by the deformations of the curve germ $C_0$ along
	each of the curves $C_j$. This gives us the desired $1$-dimensional family of
	deformations of the rational curve $g \colon \P^1 \to B$.
\end{proof}

We can now finish the proof of \Cref{prop:Hwang} as follows. Let $\Delta$ be the unit
disk, and represent the $1$-dimensional family of deformations of $g \colon \P^1 \to B$
by a holomorphic mapping $h \colon \Delta \to M_0$ with $h(0) = [g]$. Let
us denote by $g_t = G(h(t), -) \colon \P^1 \to B$ the rational curves in this
family, and by $\gt_t = \Gt(h(t), -) \colon R \to Z$ the resulting family of
morphisms from $R$. By \Cref{lem:unbranched}, the composition
\[
	\pi \circ \gt_t \colon R \cap \gt_t^{-1}(\cC) \to \pi \bigl( \gt_t(R) \cap \cC \bigr)
\]
is unramified for every $t \in \Delta$. Since the source and the target stay fixed,
it follows that $\pi \circ \gt_t$ is independent of $t$. This means that if we let
$y_1, \dotsc, y_r$ be the points in the fiber $p^{-1}(x)$ over a fixed point $x \in
(\P^1)^{\circ}$, and set $z_j = \gt(y_j)$, then as $t \in \Delta$ varies, the image
$\gt_t(y_j)$ sweeps out a small open subset of the curve $P_{z_j}$. Because $p \circ
\gt = g \circ p$, it follows that $p(P_{z_1}) = \dotsb = p(P_{z_r})$; but then
\Cref{lem:Hwang-intersection} gives $P_{z_1} = \dotsb = P_{z_r}$, and this clearly
contradicts our  initial assumption that $r \geq 2$.

\section{Proof of the main theorem}\label{sect proof}

At this point, we can prove the main theorem very quickly by putting together the
results from the previous five sections.

\begin{proof}[Proof of \Cref{main}]By way of contradiction, assume $B$ is not
	isomorphic to $\P^n$.  By \Cref{thm:Mori}, there must be a minimal degree rational
	curve $g_0:C_0\to B$ meeting $B^\circ$ with $d+1 = -\deg g_0^*\shT_B\leq n$.  By
	\Cref{prop:semipositive}, since $g_0^*\shT_B$ contains a copy of $\shT_{C_0}\cong
	\shO(2)$, it follows that $g_0^*\shT_B$ has a trivial factor, so by
	\Cref{prop:fixed-part} and \Cref{lem:splitting}, we obtain a nontrivial
	splitting of the variation of Hodge structure $V_Z$.  A general curve $g \colon C
	\to B$ meets the boundary transversely, and since $b_2(B)=1$, it must meet every
	irreducible component of
	the boundary that has dimension $n-1$. By \Cref{prop:Hwang} it follows that
	$p \colon Z\to B$ is unramified in codimension 1, and because $Z$ is normal by
	construction, $p$ must be an \'etale cover by purity of the branch locus
	\cite{purityofbranch}.  Since $B$ is simply connected, it follows that $p$ is an
	isomorphism; this means that there is a nontrivial splitting of $V$ itself, which
	contradicts Voisin's result in \Cref{thm:voisin}.
\end{proof}

\begin{remark}\label{example}The fact that $M$ was a full deformation space of a
	family of rational curves admitting a splitting was crucial to the argument.
	Indeed, consider $S\to \P^1$ an elliptic K3 surface, and $X=S^{[n]}\to
	B=\P^n$ the Hilbert scheme of $n$ points.  Then taking the cover
	$\P^1\times\P^{n-1}\to \P^n$ obtained by quotienting $(\P^1)^n$ by $S_{n-1}$, we
	have a covering family of rational curves (namely the $\P^1$ fibers) along which
	the pullback of $V$ has a fixed part (of rank $2n-2$), and the cover
	$\P^1\times\P^{n-1}\to\P^n$ splits completely over the open part of these curves.
	These curves are all tangent to the diagonal of $\P^n$ however, so this does not
	imply $\P^1\times\P^{n-1}\to\P^n$ is unramified (and indeed it is not), and it
	also does not imply that the splitting descends (and indeed it does not).
\end{remark}

\begin{remark}
The proof of \Cref{main} works equally well for a Lagrangian fibration $f:X\to B$ of
a primitive symplectic variety $X$, provided we additionally assume $X$ is simply
connected and that the discriminant is divisorial. 
\end{remark}

\def\MR#1{}
\bibliographystyle{amsalpha}
\bibliography{lagrangian}
\end{document}